\documentclass[a4paper,12pt, legno]{article}
\usepackage{amsmath}
\usepackage{amsfonts}
\usepackage{amsthm}
\usepackage{mathrsfs}
\usepackage[T1]{fontenc}
\usepackage{amsmath}
\usepackage{amsfonts}
\usepackage{amsthm}
\usepackage{mathrsfs}
\usepackage{amssymb}
\usepackage{pst-node}
\usepackage{tikz-cd}

\newtheorem{thm}{Theorem}

\newtheorem{lem}[]{Lemma}

\newtheorem{rem}[]{Remark}

\begin{document}

\title{On a limit theorem for a non-linear scaling}

\author{ Zbigniew J. Jurek (University of Wroc\l aw)}

\date{March 6, 2021}

\maketitle
\begin{quote} \textbf{Abstract.} In this note we proved that weak limits, of sums of  positive, independent and  identically distributed random variables which are re-normalized by a non-linear shrinking transform $\max(0, x-r)$,  are either degenerate or (some)  compound Poisson distributions.



\medskip
\emph{Abbreviated title:   On a limit  theorem for a non-linear scaling }

\end{quote}
\maketitle

Addresses:

Institute of Mathematics, University of Wroc\l aw,
Pl. Grunwaldzki 2/4

50-384 Wroc\l aw,  Poland; \ \  \   www.math.uni.wroc.pl/$\sim$zjjurek ;

e-mail: zjjurek@math.uni.wroc.pl

\vspace{2 cm}
In a probability theory  often one describes behaviors of normalized partial sums of random variables or vectors. Most often a normalization of partial sums is done be affine (linear) transforms thus a normalization is applied directly to each summand; cf. Jurek and Mason(1993) or Meerscheart and Scheffler (2001). In this note positive random variables are normalized by  a non-linear shrinking transformation (s-operation) and then partial sums are computed.  A  limit theorem for those sums is proved for non-negative independent and identically distributed variables; Theorem 1. The Laplace transform and some functional equations  are the  main tools in  proofs. [See also a historical note on the s-operations  in Section 6.]

\newpage
\textbf{1. A non-linear scaling $U_r$ and a theorem.}

For $r>0$, let us define  a non-linear scalings (shrinking mappings, in short: s-operations)
\begin{equation}
U_r: [0,\infty) \to [0,\infty) \ \  \mbox{by} \ \ U_r(x):= \max\{0, x-r\}\equiv (x-r)_+\,
\end{equation}
which have, among others, the one-parameter semigroup property:
\[
U_r(U_s(x))=U_{r+s}(x) \ \mbox{and} \ \  |U_r(x)-U_s(x)|\le |r-s|,  \ \ \mbox{for} \  0 \le x,r, s< \infty.
\]
If $X$ is a non-negative random variable on a fixed probability space $(\Omega, \mathcal{F}, P)$ then $U_r(X)$
 may model a situation when one receives  only an  excess  above some positive level $r$. Similarly,  such formula appears in  mathematical finance as European call options where $X$ represents a price of a stock, $r$ is a strike price, and $U_r(X)$ is an expected investor's gain. Hence, sums $S_n$,  in (2)  below,  may be interpreted as a gain from  portfolio of $n$ European  call options with the same strike price $r_n$.

\medskip
In this note we prove  the following:
\begin{thm}
Let  a sequence  $0< r_n \to   \infty$ and  $X_1, X_2,...,X_n, ...$ be a sequence of non-negative  independent identically distributed random variables. Then
\begin{equation}
S_n: = U_{r_n}(X_1)+  U_{r_n}(X_2)+...+ U_{r_n}(X_n) \Rightarrow S , \  \  \mbox{as} \  \ n\to \infty,
\end{equation}
if and only if either
\begin{equation}
(i) \ \  \emph{ there exists a constant} \ \  c>0 \  \emph{ such that}  \  S=c  \ \emph{with P.1},
\end{equation}
or
\begin{equation}
(ii) \ \  \emph{there exists an $a>0$ such that }\ \   S\stackrel{d}{=}e^{-a}\,\sum_{k=0}^\infty\, \frac{a^k}{k!}\, \nu^{\ast k},
\end{equation}
where  $\nu$  is an exponential probability distribution  with  a probability  density
$f(x):=\lambda e^{-\lambda x}\,1_{[0,\infty)}(x) \, , \ (\lambda>0)$.

In terms of Laplace transforms we have  either $\mathfrak{L}([S;t]=\exp(- ct)$ or  $\mathfrak{L}[S;t]=\exp a[\frac{\lambda}{\lambda +t}-1]$,  for $t\ge 0$.
\end{thm}
In (2)  $" \Rightarrow"$  denotes a  weak convergence of distributions and in (4)  $"\stackrel{d}{=}"$ means equality of  probability distributions. A measure in (4) is  a particular compound Poisson measure $e(m)$, where $m: =a\nu$.

\medskip
 \textbf{2. Laplace transforms and  the sufficiency part of Theorem 1.}

\noindent For  study of the weak convergence of non-negative random variables it is convenient to use their Laplace transforms. In that case,  weak convergence of measures is  equivalent to a point-wise convergence of  their Laplace transforms; for instance  cf. Feller (1966), Theorem 2a, p. 410.

Note that  (1)  and a simple calculation give  a (cumulative)  probability distribution function of $U_r(X)$:
\begin{equation}
P(U_r(X)\le x)= P(X\le x+r) \  \ \mbox{for} \ \   \ x\ge 0.
\end{equation}
Note that at zero there is jump of a size $F(r)$.

Recall that the Laplace transform of a random variable $Y$ is defined as

\noindent $ \mathfrak{L}[Y; t] := \mathbb{E}[e ^{-t Y}], \  \mbox{for } \ t\ge0. $
Hence
\begin{equation}
\mathfrak{L}(U_r(X);t) = 1 -\int_{[0,\infty)}[1-e^{-tx}]dF(x+r); \ \ F(x):=P(X\le x), \ \  x\ge 0
\end{equation}
This is so, because by (1) and (5) we get
\begin{multline*}
\mathfrak{L}(U_r(X); t)]= F(r) +  \int_{(0,\infty)}e^{-tx}dF(x+r) =
1- \int_{(r,\infty)}[1- e^{-t(y-r)}]dF( y)  \\ =1- \int_{(0,\infty)}[1- e^{-tx}]dF( x+r) =  1- \int_{[0,\infty)}[1- e^{-tx}]dF( x+r) .
\end{multline*}
[Note that the function under last integral sign vanishes at zero].

\medskip
\noindent Hence, from (6), for  $r_n>0$  and  independent  identically distributed  variables $(X_k), k=1,2,...,n$  with  a  distribution function $F$ , we have
\begin{multline}
\mathfrak{L}(S_n; t)= ( \mathbb{E}[e^{ -t\,U_{r_n} (X_1)}])^n   = \Big(  1 - \int_{[0,\infty)} [ 1- e^{-tx}] dF(x+r_n)\Big)^n  \\ =  \Big(  1 -  \frac{1}{n}\, [\, \int_{[0,\infty)}
   \frac{1- e^{-tx}}{1- e^{-x}} dG_n(x)\,]\Big)^n ,   \  \mbox{where distributions  $G_n$ are defined as}\ \\
G_n(y):= n \int_0^y (1-e^{-x})dF(x+r_n),   \ y \ge 0;\ \  \,\, G_n(\infty)\le n(1-F(r_n)).
\end{multline}
Thus $(G_n)$ are  continuous distribution function of finite measures on $[0,\infty)$.

\medskip
 Here are  examples of normalizing   sequences  $(r_n)$  and  random variables$(X_n)$  such that  $S_n \Rightarrow S $ with limits  $S$    described in Theorem 1 by  formulas   (4) and (3), respectively.

\medskip
\textbf{Example 1.} Let $(X_n)$ be i.i.d. exponentially distributed with probability density $\lambda e^{-\lambda x} 1_{[0,\infty)}(x)$ and for given $a>0$,  let us choose $r_n>0$ such that
$e^{- \lambda r_n}= a/n,$
for all n such that $a/n < 1$.
Then by (7), the  limit  of  corresponding distributions $(G_n)$ is  as follows
\begin{multline*}
\lim_{n\to \infty} G_n(y) =\lim_{n\to \infty} n\int_0^y(1-e^{-x})\lambda e^{- \lambda (x+r_n)}dx \\ = a \int_0^y(1-e^{-x})\lambda e^{- \lambda x}dx = G(y),  \ \  \ \ \mbox{for}     \ \   0\le y<\infty. \qquad  \qquad \qquad
\end{multline*}
Finally,  we have the Laplace transform of  a  limit $S$:
\begin{multline*}
\mathfrak{L}[S;t]= \exp\int_0^\infty \frac{1-e^{-tx}}{1-e^{-x}}G(dx) \\ = \exp(-a\int_0^\infty (1-e^{-tx})\lambda e^{-\lambda x}dx= \exp a[\frac{\lambda}{\lambda +t}-1].
\end{multline*}
Here we recognize that $\exp [ \frac{\lambda}{\lambda +t}-1]$ is a Laplace transform of  compound Poisson variable
$ \sum_{k=0}^{N_{1}} \,Y_k$, where $(Y_n)$  are i.i.d.  exponentially distributed with parameter $\lambda$ and independent of a  Poisson variable $N_{1}$.

\medskip
\medskip
\textbf{Example 2.} Let random variables  $ (X_n)$ be independent with the same (cumulative) distribution function $F$ which has the probability density

\noindent $(\frac{2}{\pi})^{1/2}\exp(-x^2/2)1_{(0,\infty)}(x)$. Note that a function $ k(x):= (\frac{2}{\pi})^{1/2}x^{-2}e^{-x^2/2}$ on $(0,\infty)$ is continuous and strictly decreasing  from  infinity to zero. Thus for a given $c>0$ we may choose $r_n$ as the solution of the equation $ k(x)= c/n$, that is,
$ n \, (\frac{2}{\pi})^{1/2}(r_n^2\, e^{r_n^2/2})^{-1}= c$.

Taking into account the above, from (7), $G(0):=\lim_{n\to \infty}G_n(0)= 0$ and  for $y> 0$ we have
\begin{multline*}
G(y):= \lim_{n\to \infty} G_n(y)= \lim_{n\to \infty}(\frac{2}{\pi})^{1/2}  n \int_0^y(1-e^{-x}) \exp(- (x+r_n)^2/2)dx \\=
\lim_{n\to \infty} (\frac{2}{\pi})^{1/2} n \exp(-r_n^2/2)\int_0^y(1-\exp(-x))\exp(-x^2/2)\exp(-r_nx)dx \\
= \lim_{n\to \infty} (\frac{2}{\pi})^{1/2} n \exp(-r_n^2/2)\int_0^y\frac{1-\exp(-x)}{x} x \exp(-x^2/2)\exp(-r_nx)dx \\
\mbox{(using substitution $xr_n=:u$ we get) } \\
=\lim_{n\to \infty}(\frac{2}{\pi})^{1/2 }\,n \, (r_n^{2}\exp (r_n^2/2))^{-1}\,\int_0^{yr_n}\frac{1-\exp(-u/r_n)}{u/r_n}  \exp(- \frac{u^2}{2r_n^2})u \exp(-u)du\\= c \lim_{n\to \infty} \int_0^\infty [\, 1_{(0, yr_n]}(u)\, \frac{1-\exp(-u/r_n)}{u/r_n}  \exp(- \frac{u^2}{2r_n^2}),]\, u \exp(-u)du \\= c \int_0^\infty u \exp(-u)du =  c, \ \  \mbox{for} \ y>0.
\end{multline*}
All in all, in this example,  $G(dx)=c\delta_0(dx)$ is a distribution function of a measure concentrated at zero with a mass $c$. Consequently,
$$
\mathfrak{L}[S;t]= \exp \int_0^\infty\frac{1-e^{-tx}}{1-e^{-x}}G(dx)= \exp(-ct), t\ge 0 ,
$$
or in other words,  the limit  $S=c$ with P.1.

This completes a proof of the necessity part of Theorem 1.

\medskip
\medskip
\textbf{3. Some functional equations and their solutions.}

\medskip
From now on we assume that $S_n\Rightarrow S$ which (by (7))  is equivalent to
\begin{equation}
\mathfrak{L}(S;t)=  \exp\big(- \lim_{n\to \infty}\int_{[0,\infty)}
   \frac{1- e^{-tx}}{1- e^{-x}} dG_n(x)\big), \ \  t \ge0,
\end{equation}
where $G_n$ are distribution functions of finite measures on the positive half-line.

\begin{lem} (i)  If  $S_n \Rightarrow S$,  in (2),  then distribution functions $(G_n)$ are uniformly bounded, that is,  $\sup_{(n \ge 1)} G_n ( + \infty)\le K < \infty $.

\noindent (ii) If $G$ is   a limit point of $\{G_n\}$ then $G(\{\infty\})=0$.
\end{lem}
\begin{proof} (i)  In contrary, assume that there exists sub-sequence $n_k, \  k= 1,2,...$ such that $ \lim _{k\to \infty} G_{n_k}(+ \infty)=  \infty $ then from (6) we get that
$$\lim_{k\to \infty} \mathfrak{L}(S_{n_k}; 1)= \lim_{k\to \infty} \big ( 1-\frac{G_{n_k }(\infty)}{n}\big)^n=  0$$
which contradict the fact that Laplace transform are positive functions.

Hence distribution functions  $(G_n)$   of  finite measures on  the positive half-line and uniformly bounded (conditionally compact).

(ii) Let   $G$ on
$[0,\infty]$  be  a weak limit  of  subsequence of distribution functions $(G_n)$.  Then
\begin{multline*}
\mathfrak{L}(S; t) =\lim_{n\to \infty} \mathfrak{L}(S_n;t) =  \exp (- \int_{[0,\infty]} \frac{1- e^{-tx}}{1-e^{-x}} dG(x)) \\ = \exp (-\int_{[0,\infty)} \frac{1- e^{-tx}}{1-e^{-x}} dG(x))  \cdot \, \exp(- G(\{\infty\})).
\end{multline*}
Taking $t \to 0$ , the function under integral sign  tends to zero  and hence we must have  $1= \mathfrak{L}(S; 0)= \exp (-G(\{\infty\})$ which means that $G(\{\infty\})=0$. This completes a proof of  Lemma 1.
 \end{proof}
\begin{rem}
\emph{It may be worthy to recall that a family of probability measures on compact topological  space  is compact convolution semigroup, in weak topology; cf. Parthasarathy(1967), Theorem 6.4. p. 45.}
\end{rem}
\begin{rem}
\emph{ For a discussion below recall that for (probability) distribution functions $F_n$ and $ F$  we  have}:

\noindent \emph{(a)  if $F_n \Rightarrow F, \  a_n \to a$ , $a$ is  continuity point of $F$,  then $F_n(a_n)\to F(a)$;}

\noindent \emph{(b) if $F_n$ converges to $F$ point wise and $F$ is continuous then $F_n \to F$ uniformly.}
\end{rem}

From Lemma 1 and the limit in (7) we have that  $G_n \Rightarrow G$ (weak convergence)  to a distribution function of finite (not necessarily probability) measure $G$. To identify $G$ we introduce auxiliary functions:
\begin{multline}
H_n(u):= \int_1^u \frac{1}{1-e^{-x}}dG_n(x); \ \ H_n(u)\to  \ H(u):= \int_1^u \frac{1}{1-e^{-x}}dG(x);  \\  \  \ 0<  u<\infty, \ \
\mbox{and hence} \ \  \ G(u) - G(1) = \int_1^u (1- e^{-x})dH(x).  \quad
\end{multline}
In particular from Remark 2, we have that  if $u$ is a continuity point of $H$ (note that $G$ may have a jump at $u$)  and $u_n\to u$ then  $H_n(u_n)\to H(u)$.

\medskip
\noindent From  (7),  for $   1\le u<\infty $,   we have that
\begin{equation}
 H(u)= \lim_{n\to \infty}H_n(u) = \lim_{n\to\infty} n[F(u+r_n)-F(1+r_n)]
\end{equation}
\begin{lem}
Assume that $ w_n:=r_{n+1}-r_n \to w \in [0,\infty]$ ( or choose a sub-sequence). Then
\begin{equation}
H(u+w)=H(u)+ H(1+w) ,  \ 0<u < \infty
\end{equation}
\end{lem}
\begin{proof}
Since $r_{n+1}= w_n +r_n $, from (10), we infere that
\begin{multline*}
H(u)= \lim_{n\to \infty} (n+1) [F(u+r_{n+1})-F(1+r_{n+1}] \\ =
\lim_{n\to \infty} \Big[
 n\,[F(u+ w_n+r_n)- F(1+w_n+r_n)]  +  [F(u+r_{n+1})-F(1+r_{n+1}] \Big] \\
 =\lim_{n\to \infty} \big( n\,[F(u+ w_n+r_n)-  F(1+r_n)] - n [F(1+w_n+r_n)-F(1+r_n)]  \big)\\=
H(u+w)-H(1+w), \ \ \mbox{for} \  0 < u <\infty  \ \mbox{and} \ \  0\le w\le \infty. \qquad
\end{multline*}
\end{proof}
What are possible functions $H$ satisfying  a functional equation (11)?   We consider the three possible cases for a limit $w$.
\begin{lem}
 If $ w_n:=r_{n+1}-r_n \to w \in ( 0,\infty]$ (or choose a subsequence) then  $H(u)=0$ and  $H(u)=H(\infty)$ are  the only solutions to a  functional equation $ H(u)+H(1+w)=H(u+w)$, where $0<u<\infty$.
\end{lem}
\begin{proof}
A case $w=\infty$  is obvious as $H(\infty)$  is finite so  we get $H(u)=0$.

\noindent Let $0<w<\infty$. Then by the induction argument we have
\[
H(u+kw)=H(u)+k H(1+w), k\ge 1
\]
Since left hand side has a finite limit $H(\infty)$ as $ k\to \infty$ therefore $H(1+w)=0$.  Consequently, $H(u+kw)=H(u), k\ge 1$ which gives $H(u)=H(\infty)=const $ and completes a proof.
\end{proof}

\medskip
For the reaming case $w=0$  (i.e.,  $\lim_{n\to \infty} w_n =0$; cf. Lemma 2) we need the following elementary fact.

\medskip
\textbf{Fact.} If $0<r_n \nearrow \infty$ and $r_{n+1}-r_n \to 0$ as $ n\to \infty$ the for each  number $0<c<1$ there exists a sub-sequence $(k_n)$  such that $r_{k_n}-r_n \to c$.
\begin{proof}
Choose explicitly  $k_n:=\sup\{k\ge n: r_k-r_n<c\}$. Then we have that  $r_{k_n}-r_n <c \le r_{k_n +1}-r_n$ and  a length of an interval containing $c$  is $r_{k_n+1}-r_{k_n} \to 0$ as $n\to \infty$, which completes an argument.
\end{proof}
\begin{lem}
If  $r_{n+1}-r_n \to 0$ then for any number $0<a<1$ $a$  there exists $b\ge1$ (depending on $a$)  such that
\begin{equation}
H(u+a)= b^{-1}H(u) +H(1+a),  \ \ 0<u<\infty.
\end{equation}
\end{lem}
\begin{proof} From previous Fact, there exist a sub- sequence $(k_n)$ such that $ v_n:=r_{k_n}-r_n\to a$. Furthermore,
\begin{multline*}
H(u)=\lim_{n\to \infty} k_n[F(u+r_{k_n})-F(1 + r_{k_n})]\\
=\lim_{n\to \infty}\frac{ k_n}{n} \big[ n[F(u+v_n+r_n)- F(1  + r_n)] -   n [F( 1+v_n+r_n)-F(1+r_n)] \big].
\end{multline*}
Since  limits in  square brackets [ ...] exist and  are $H(u+a)$ and $H(1+a)$, respectively, we infer that $b:=\lim_{n
\to \infty}\frac{k_n}{n}$ exists. Finally,
we get a formula $H(u)=b\,H(u+a)- H(1+a)$,
 which completes  a proof of the  lemma.
\end{proof}
\begin{rem}
Note that for $b=1$ the equation (10) coincides with the one  given in Lemma 3.
\end{rem}
To solve the functional equation (10) we need some facts that are quoted in the Appendix and from them we have:

\begin{lem}
The only solutions to an equation (10) in Lemma 4 are
\begin{equation}
H(u)= const  \ \mbox{or} \ \ H(u)= \alpha( e^{-\gamma u} - e^{- \gamma}),  \  \alpha<0, \  \gamma >0.
\end{equation}
\end{lem}
\begin{proof}
Using  Appendix for

$f(u):= H(u), \psi(y):=H(1+y), \ \mbox{and $\phi$ such that} \  \phi(a):=b^{-1},$

\noindent we retrieve our equation (11). Thus  from (12),  $H(u)=const $ if $f_1(x)=0$.

In a case $f_1(x)=\alpha e^{\gamma x}$ we get $e^{\gamma a}=b^{-1} <1$  which implies  $\gamma <0$ and  that gives
\[
 H(u)=\alpha e^{\gamma u}+\beta\  \mbox{with restrictions} \,   H(1)=0; \,  H(u)<0 , \ \mbox{for} \   0<u<1;
\]
From first restriction we get $H(u)=\alpha(e^{\gamma u}-e^\gamma)$. Thus the second restriction imposes that $\alpha<0$ which completes arguments for (13).
\end{proof}

\textbf{4. Proof of the necessity part of Theorem 1.}

From Lemma 2, we have  that
$G(y)-G(1)=\int_1^ y(1-e^{-x}) dH(x), y>0.
$
and  from Lemma 5 we know forms of   $H$ ( introduced in Lemma 2). Consequently,
we have  that either $G(y)=c>0$ (constant) for $ y>0$ and $G(0)=0$ or
\[
G(y)=G(1)+ \int_1^y (1-e^{-x})(-\alpha \gamma)e^{-\gamma x}dx;   \  \ G(0)=0.
\]
This may be a written as
$$
G(y)= - \alpha \gamma \int_0^y (1-e^{-x})e^{-\gamma x} dx;  \ \ y\ge 0;   \ \alpha <0, \ \gamma >0.
$$
Using (7)  we have
\[
\mathfrak{L}(S;t)= \exp(- ct) \ \ \ \ \mbox{or} \ \ \mathfrak{L}(S;t)=\exp (-\alpha)[\frac{\gamma}{t+\gamma}-1],  \  \alpha<0, \ \gamma>0,
\]
which completes a proof of Theorem 1.

\medskip
\textbf{5. Appendix.}

For  ease of reference let us quote  the following fact.
\begin{lem}
Let functions $f$, $\phi$ and  $\psi$ be defined  on  real (or positive) numbers  and satisfy  equations
\begin{equation}
(i)  \ \ f(x+y)=\phi(y)f(x)+ \psi(y) \ \ \mbox{or} \ \  (ii) \
\  f(x+y)= f(x)+ \psi(y)
\end{equation}
If $f$ is continuous (or bounded on a set of positive  Lebesque measure) then  solutions to equation (14), part (i), are
\[
f(x)= \alpha e^{\gamma x} +c, \  \phi(x)=e^{\gamma x},\   \psi(x)= c[1-e^{\gamma x}];   \ (\alpha, \gamma, c \ \mbox{are constants}).
\]
and for part (ii) solutions are
\[
f(x)=\gamma x +c, \  \  \psi(x)=\gamma x , \ \  \ \ (\gamma, c \ \mbox{are constants});
\]
\end{lem}
\noindent cf. Aczel (1962), Theorem 8 on  p. 22 or  Aczel (1966), Chapter 3, Theorem 1 on p.150.

\medskip
\textbf{6.  A historical note.}
A problem of the  characterization of  all possible limit distributions for sequences  $U_{r_n}(X_1)+  U_{r_n}(X_2)+...+ U_{r_n}(X_n) $ for \underline{positive} variables  was posed by Kazimierz  Urbanik around 1971. (For many years the notation $T_r$ instead of $U_r$ was used.)
Proofs presented in this  paper  are taken from MS Thesis [5] written in Polish language in 1972 and they were never publish before in English. References [3] and [7]-[9] are added to this article to point out the current state of the problem.

\medskip
On the other hand, the Reader should be aware that
the definition of shrinking s-transforms $U_r$ have a natural generalization to vector valued random variables $X$ as follows:  $U_r(0):=0$ and
$$
U_r(X):=\max(||X||-r, 0)\frac{X}{||X||}, \ X \neq 0; \  \ \mbox{where $||\cdot ||$ denotes a norm.}
$$
Urbanik problem for  a Hilbert space valued, not necessarily identically distributed,  random vectors  was  completely  solved in 1977 where in the proof the Choquet Theorem on extreme points in convex compact sets was used;  cf. Jurek (1981).

However, knowing solutions for real random variables it is not obvious how to infer characterizations for positive variables. Furthermore, Hilbert space generality may not be accessible for a wider audience. Those concerns
led us to the decision to present here elementary proofs for non-negative  random variables from [5].

Let us also add here that more recently, in Bradley and Jurek (2015), stochastic independence of variables in Theorem 1 was replaced by the weak dependence (strong mixing).

Last but not least, potential applications in a mathematical finance (mentioned in the Introduction) of which the Author was not aware 50 years ago,  may generate some additional interest for this quite straightforward and elementary proofs.

\medskip
\medskip
\textbf{ Acknowledgments.} Reviewer's comments lead to a better presentation of some arguments and  improved the language of this note.

\medskip
\textbf{References.}

\medskip
[1] \ J. Aczel (1962), \emph{On applications and  theory of functional equations},

Birkhauser Verlag,  Basel  1969.

\medskip
[2] \ J. Aczel (1966), \emph{Lectures on functional equations and their}

\emph{applications},  Academic Press, New York.

\medskip
[3] \  R. C. Bradley and Z.J. Jurek (2015), On a central limit theorem

for shrunken weakly dependent random variables, \emph{Houston Journal of}

\emph{ Mathematics}, vol. 41, No. 2, pp. 621-638.

\medskip
[4] \  W. Feller(1966), \emph{An Introduction to Probability Theory},  vol. 2,

J. Wiley $\&$  Sons, New York 1966.

\medskip
[5] \ Z. J. Jurek (1972), On a class of limit distributions, University of

Wroclaw, MS Thesis (in Polish), University Archives, call number:

\emph{W IV-5200/Jurek Zbigniew;} (To Archives  e-mail address: auw@uwr.edu.pl)

\medskip
[6] \  Z. J. Jurek (1981), Limit distributions for sums of shrunken random

variables, \emph{Dissertationes Mathematicae},vol. CLXXXV, Polskie

Wydawnictwo Naukowe, Warszawa, 1981,  pp. 50;

(accepted for a  publication on  November 29, 1977).

\medskip
[7] \ Z. J. Jurek and J. D. Mason (1993), \emph{Operator- limit distributions in}

\emph{ probability theory}, J. Wiley $\&$ Sons, Inc. New York

\medskip
[8] M. M. Meerscheart and H-P Schefller (2001), \emph{Limit distributions for}

\emph{sums of independent random vectors},  J. Wiley$\&$ Sons, Inc. New York.

\medskip
[9] \  K. R. Parthasarathy (1967), \emph{Probability measures on metric spaces},

Academic Press, 1967.

\end{document}